\documentclass[12pt]{amsart}

\usepackage{amsmath,color,tikz-cd}
\usepackage{amsthm}
\usepackage{hyperref}
\usepackage[margin=1.0in]{geometry}
\usepackage{tikz}
\usetikzlibrary{arrows.meta}
\usepackage{comment}
\numberwithin{equation}{section}

\newtheorem{prop}{Proposition}
\newtheorem{lemma}[prop]{Lemma}

\newtheorem{thm}[prop]{Theorem}

\numberwithin{prop}{section}

\theoremstyle{definition}

\newtheorem{rmk}[prop]{Remark}

\newcommand{\brs}[1]{\left| #1 \right|}

\newcommand{\gD}{\Delta}

\newcommand{\gl}{\lambda}

\newcommand{\N}{\nabla}

\newcommand{\IP}[1]{\left<#1\right>}

\DeclareMathOperator{\Spin}{Spin}

\DeclareMathOperator{\Tor}{Tor}

\DeclareMathOperator{\Int}{Int}

\newcommand{\R}{\mathbb R}

\newcommand{\Aroof}{\hat{\mathcal A}}

\renewcommand{\bar}[1]{\overline{#1}}

\begin{document}

\title[Topology of low-dimensional string backgrounds]{Topology of low-dimensional generalized Ricci solitons and string backgrounds}

\author{Jeffrey Streets}
\address{Rowland Hall\\
         University of California, Irvine\\
         Irvine, CA 92617}
\email{\href{mailto:jstreets@uci.edu}{jstreets@uci.edu}}

\date{\today}

\begin{abstract} Adapting ideas of \cite{akutagawa2007perelman}, we show that compact generalized Ricci solitons (GRS) have positive Yamabe invariant.  We observe a Cheeger-Gromoll-type splitting theorem for GRS as a corollary of the splitting theorem for Bakry-\'Emery Ricci curvature in \cite{wei2009comparison}.  Using this we show that low dimensional GRS are diffeomorphic to $S^3 / \Gamma$ or $S^3 \times S^1 / \Gamma$.  We determine various topological constraints on string backgrounds (Bismut-Hermitian-Einstein (BHE), strong torsion $G_2$, strong torsion $\Spin(7)$-manifolds) and show in most cases that they cannot exist on the same manifolds as their classical special holonomy counterparts.  Finally we determine the topology of BHE threefolds under natural constraints, relying on an extension of parts of Kollar's characterization of Seifert fibered $5$-manifolds over complex orbifolds \cite{kollar2006circle}.
\end{abstract}

\maketitle

\section{Introduction}

Generalized Ricci solitons arise as critical points of the string effective action \cite{Friedanetal}, and describe the geometry of string backgrounds \cite{Strominger1986}, defined here as Riemannian metrics admitting a compatible connection with closed skew symmetric torsion and reduced holonomy.  This is a subject of intense recent activity (e.g. this incomplete list from only the last few years:  \cite{apostolov2025rigidity,apostolov2025toric,apostolov2026pluriclosed,barbaro2023bismut,brienza2026structure,carmona2026cylindrical,fino2025some,ivanov2023riemannianspin7,ivanov2023riemannianG2,JFS,lauret2023bismut,lee2025dynamical,podesta2023bismut,podesta2024infinite,podesta2025three,womack2026classification}).  Despite the deep roots of this subject and the recent activity, little is known about the topology of such manifolds.  We build some some general theory in this direction and use it to prove definitive topological classification results.

To begin we investigate the topology of generalized Ricci solitons in all dimensions.  First, we make the elementary observation that $b_3 \geq 1$.  Next, inspired by \cite{akutagawa2007perelman}, we show the Yamabe invariant of the conformal class is nonnegative.  We further observe that the splitting result for Bakry-\'Emery Ricci curvature \cite{wei2009comparison} naturally splits generalized Ricci solitons, leading to a Cheeger-Gromoll-type splitting theorem (cf. Theorem \ref{t:splitting}).  We combine these tools to prove a topological classification in dimensions $3$ and $4$.  Note that in the statements below the adjective \emph{nontrivial} indicates that $H \neq 0$.
\begin{thm} \label{t:lowdGRStopology} 
The following hold:
\begin{enumerate}
\item If $(M^3, g, H, f)$ is a nontrivial compact GRS, then $M^3 \cong S^3 / \Gamma$.
\item If $(M^4, g, H, f)$ is a nontrivial compact GRS, then $M^4 \cong S^3 \times S^1 / \Gamma$.
\end{enumerate}
\end{thm}
\noindent We furthermore determine various topological characteristics of string backgrounds, including vanishing results for the Dolbeault cohomology of Bismut-Hermitian-Einstein (BHE) manifolds, and vanishing of all top degree characteristic classes on strong torsion $\Spin(7)$-manifolds.  These results in particular show that in these cases, as well as strong torsion $G_2$ structures, the moduli space is disjoint from their classical special holonomy counterparts (cf. Remark \ref{r:isolatedmoudli}).

We then turn to a more refined investigation of the topology of BHE threefolds.  As shown in \cite{JFS}, non-K\"ahler BHE come equipped with a canonical rank $2$ distribution of Bismut-parallel vector fields.  The simplest possible case is when these vector fields generate a principal $T^2$ action, so the quotient space is smooth, and we refer then to $M$ as regular.  Our next main theorem classifies the possible topologies in this setting: 

\begin{thm} \label{t:BHE3foldssmooth} Let $(M^6, g, J)$ be a compact non-K\"ahler regular BHE manifold.  Then a finite cover of $M$ is diffeomorphic to $N \times T^k$ where either
\begin{enumerate}
    \item $k = 0$ and $N \cong \#^{r-1} S^2 \times S^4 \#^r S^3 \times S^3$, \ $1 \leq r \leq 8$,
    \item $k = 1$ and $N \cong \#^9 S^2 \times S^3$,
    \item $k = 3$ and $N \cong S^3$.
\end{enumerate}
\end{thm}
\noindent The proof starts from the Cheeger-Gromoll splitting discussed above.  The cases $k \geq 2$ are treated by Theorem \ref{t:lowdGRStopology}.  In case $k = 1$, we exploit the Smale-Barden classification \cite{barden1965simply, smale1962structure}, with the main point to show that the second homology of the split five-manifold is torsion-free if the quotient satisfies this property.  The case $k = 0$ is treated using the Wall-Jupp classification \cite{jupp1973classification, wall1966classification}, where again the main issue is to show that the second homology is torsion-free if that of the quotient is torsion-free.  With these in place the result essentially follows from the analysis of smooth quotients of \cite[Lemma 2.3]{apostolov2026pluriclosed}, obtained via Kodaira classification, with one refinement from a Bochner formula (cf. Proposition \ref{p:BHEtopology}) which rules out ruled surfaces over an elliptic curve.

Building on the ideas behind Theorem \ref{t:BHE3foldssmooth}, we next consider the case when the symmetries generate a locally free $T^2$ action, and refer then to $M$ as quasiregular.  Recent breakthrough work on BHE threefolds \cite{apostolov2026pluriclosed} exhibited such examples with $k = 1$ whose quotient has orbifold singularities at divisors of genus zero.  Given the deep links of this problem to toric geometry, it is natural to conjecture that the orbifold locus always consists of rational curves.  Our next main theorem classifies the topology of BHE threefolds with infinite fundamental group under this hypothesis:

\begin{thm} \label{t:BHE3foldskollar} Let $(M^6, g, J)$ be a compact non-K\"ahler quasiregular BHE manifold with $\brs{\pi_1(M)} = \infty$, and suppose the orbifold divisors all have genus zero.  Then a finite cover of $M$ is diffeomorphic to $N \times T^k$ where either
\begin{enumerate}
    \item $k = 1$ and $N \cong \#^r S^2 \times S^3,\ r \geq 1$,
    \item $k = 3$ and $N \cong S^3$.
\end{enumerate}
\end{thm}
\noindent As in the proof of Theorem \ref{t:BHE3foldssmooth}, the main point is to show that $H_2(N)$ is torsion-free, which is provided by the work of Kollar \cite{kollar2006circle} using our structural hypothesis on the orbifold.  

Finally we turn to the case of effective $T^2$ action and finite fundamental group.  Again nontrivial examples of this type were constructed in \cite{apostolov2026pluriclosed}, and again the quotient has orbifold divisors of genus zero meeting at orbifold points.  Under this hypothesis and a further structural assumption on the stabilizers discussed below, we classify the topology:
\begin{thm} \label{t:BHE3foldsrank0} Let $(M^6, g, J)$ be a compact non-K\"ahler quasiregular BHE manifold with $\pi_1(M) = 0$.  Assume the orbifold divisors have genus zero and that any intersecting divisors have coprime stabilizer subgroups.  Then $M \cong \#^{r-1} S^2 \times S^4 \#^r S^3 \times S^3$.
\end{thm}
\noindent Again the key point is to show that $H_2(M)$ is torsion-free, hence we want to find an extension of the computation \cite{kollar2006circle} of the torsion of Seifert fibrations over complex orbifolds used in Theorem \ref{t:BHE3foldskollar}.  A key issue is to show that the orbifold is cyclic, a point which holds automatically for Seifert circle fibrations.  For BHE threefolds in general all stabilizers are a priori cyclic except possibly at divisor intersections, which here follows from the assumption that stabilizer subgroups of the divisors have coprime orders (cf. Lemma \ref{lem:cyclic-intersection-stabilizer}).  We note that this hypothesis holds in the known examples of \cite{apostolov2026pluriclosed}.  With this cyclicity in hand, we determine the torsion of $H_2(M)$ via a spectral sequence computation for an equivariant CW decomposition on $M$.

\begin{rmk} 
\begin{enumerate}
\item In general for non-K\"ahler BHE threefolds, when $k = 2$, the transverse space $N$ is isometric to a Bismut-flat metric on a standard Hopf surface, which splits a further $S^1$ factor, hence the original splitting was not maximal and in fact $k = 3$.  The cases $k \geq 4$ are easily shown to be Riemannian flat, hence K\"ahler, and ruled out by assumption.
\item In \cite{grantcharov2008calabi}, the authors constructed pluriclosed metrics, and distinct Hermitian (non-pluriclosed) metrics with vanishing Bismut Ricci form, on all of the manifolds $\#^{r-1} S^2 \times S^4 \#^r S^3 \times S^3$.  In fact these metrics are constructed via a bundle construction, and are invariant under a free $T^2$ action.  Theorem \ref{t:BHE3foldssmooth} shows that for $r > 8$ we cannot construct a BHE invariant under this action.  An interesting question then is to determine the long-time behavior of invariant pluriclosed flow \cite{garcia2026pluriclosed,PCF} in this setting.  Our results show that smooth convergence cannot be expected, and rather one may expect orbifold singularities to form at infinity.
\item The manifolds in case (1) of Theorem \ref{t:BHE3foldskollar} admit natural complex structures arising from contact structures on $\#^k S^2 \times S^3$.  However, we are not aware of a general construction of pluriclosed metrics on these spaces.  One naturally would like to impose furthermore that the $S^1$ factor splits isometrically.  Given this, constructing a pluriclosed metric naturally adapted to a contact structure on $\#^k S^2 \times S^3$ requires Chern class of the contact structure to square to zero, which in general will require orbifold singularities on the quotient.
\item A famous folklore conjecture of Yau asserts that there are only finitely many diffeomorphism types of K\"ahler Calabi-Yau threefolds (cf. e.g. \cite{wilson2021boundedness}).  In fact this is conjectured to hold in all dimensions.  As stated such a conjecture is obviously false for BHE even in dimension $2$ as evidenced by lens space quotients of the standard Hopf surface.   However we can ask a similar question for BHE with free abelian fundamental group, as every BHE is covered by such a space by Theorem \ref{t:splitting}.  In view of the recent GIT perspective on BHE threefolds \cite{apostolov2026pluriclosed}, it is plausible that there still exist infinitely many topologically distinct BHE with this restriction.  On the other hand Theorems \ref{t:BHE3foldskollar} and \ref{t:BHE3foldsrank0} give a simple numerological classification of the possible topologies, whereas there is a vast array of known K\"ahler Calabi-Yau threefolds.
\end{enumerate}
\end{rmk}

\noindent \textbf{Acknowledgments:} We thank Vestislav Apostolov and Beatrice Brienza for insightful conversations.  Aspects of the strategy for the proof of Theorem \ref{t:BHE3foldsrank0} were developed in consultation with ChatGPT.

\section{Topology of generalized Ricci solitons} \label{s:topologyGRS}
\subsection{Nontrivial third Betti number}

\begin{prop} \label{p:b3bound} Let $(M^n, g, H, f)$ be a compact generalized Ricci soliton.  Then $[H] = 0$ if and only if the soliton is trivial.  In particular nontrivial solitons have $b_3 \geq 1$.
\begin{proof}  The argument is implicit in \cite[Proposition 4.2]{ASU3}.  The given form $H$ is closed and satisfies $d^* (e^{-f} H) = 0$.  Thus if $H = db$ we have
\begin{align*}
    \int_M \brs{H}^2 e^{-f} dV_g =&\ \int_M \IP{ b, d^* (e^{-f} H)} dV_g = 0.
\end{align*}    
\end{proof}
\end{prop}

\subsection{GRS and Yamabe constants}

Taking inspriration from \cite{akutagawa2007perelman}, here we show using the Perelman-type $\gl$ quantity for generalized Ricci flow (\cite{OSW} cf. also \cite{GRFbook}), that the Yamabe invariant of a GRS is nonnegative.

\begin{lemma} \label{l:lambdasign} Let $(M^n, g, H, f)$ be a compact GRS.  Then
\begin{align*}
    \gl(g, H) = \tfrac{1}{6} \int_M \brs{H}^2 e^{-f} dV_g.
\end{align*}
In particular, $\gl(g, H) \geq 0$, with equality if and only if the soliton is trivial.
\begin{proof} Taking the trace of the GRS equation yields $R_g - \tfrac{1}{4} \brs{H}^2 + \gD f = 0$.  It follows using integration by parts
\begin{align*}
    \gl(g, H) =&\ \int_M \left( R_g - \tfrac{1}{12} \brs{H}^2 + \brs{\N f}^2 \right) e^{-f} dV_g\\
    =&\ \int_M \left( R_g - \tfrac{1}{12} \brs{H}^2 + \gD f \right) e^{-f} dV_g = \tfrac{1}{6} \int_M \brs{H}^2 e^{-f} dV_g,
\end{align*}
as claimed.  In the case of equality it follows immediately that $H = 0$, so that $(g, f)$ is a steady Ricci soliton, and the remaining claims are a standard fact.
\end{proof}
\end{lemma}

\begin{prop} \label{p:AroofGRS} Let $(M^n, g, H, f)$ be a compact GRS.  Then $Y_{[g]} \geq 0$, with equality if and only if the soliton is trivial.
\begin{proof} Choose a Yamabe minimizer $u$.  Then we obtain
\begin{align*}
    \brs{\brs{u}}_{L^{2n/n-2}}^2 Y_{[g]} =&\ \int_M \left( R_g u^2 + 4 \frac{n-1}{n-2} \brs{\N u}^2 \right)\\
    \geq&\ \int_M \left( R_g u^2 + 4 \brs{\N u}^2 \right)\\
    =&\ \mathcal F(g, H, -2 \ln u)\\
    \geq&\ \gl(g, H) \brs{\brs{u}}_{L^2}^2.
\end{align*}
Both claims now follow immediately from Lemma \ref{l:lambdasign}.
\end{proof}
\end{prop}

\subsection{Nonnegative Bakry-\'Emery Ricci and splitting}

In \cite{wei2009comparison} the classic Cheeger-Gromoll splitting theorem was extended to the case of Bakry-\'Emery Ricci curvature with bounded potential (cf. also \cite{lichnerowicz1970varietes}).  As an elementary corollary we show that this splitting naturally decomposes generalized Ricci solitons:

\begin{thm} \label{t:splitting} If $(M, g, H, f)$ is a compact generalized Ricci soliton, then $M$ is finitely covered by $N \times T^k$, where:
\begin{enumerate}
    \item $\pi_1(N) = 0$,
    \item $g = g_N + g_F$ where $g_F$ is a flat metric on $T^k$,
    \item $H = \pi_N^* H_N$ for a closed three-form on $H_N$,
    \item $f = \pi_N^* f_N$ for $f \in C^{\infty}(N)$.
\end{enumerate}
Hence, $(g_N, H_N, f_N)$ is a generalized Ricci soliton.
\begin{proof} Since a generalized Ricci soliton has nonnegative Bakry-\'Emery Ricci curvature, all of the claims except item (3) follow immediately from \cite[Theorem 6.6]{wei2009comparison}.  Given this structure, inspecting the generalized Ricci soliton equation it follows that for any Killing field $X$ tangent to the torus we have $i_X H = 0$.  Using \cite[Proposition 2.8]{kopfer2023bochner} it follows that $H$ is Levi-Civita parallel with respect to $X$, and hence item (3) follows.
\end{proof}
\end{thm}

\begin{rmk} Further structure follows immediately from [\cite{wei2009comparison} Corollary 5.7] \label{c:GRSfundamentalgroup}. In particular, compact generalized Ricci solitons satisfy
\begin{enumerate}
    \item $b_1 \leq n$.
    \item $\pi_1(M)$ has a finite index free abelian subgroup of rank $\leq n$.
    \item $b_1 = n$ or $\pi_1(M)$ has a free abelian subgroup of rank $n$ if and only if $(M, g)$ is a flat torus and $f$ is constant.
\end{enumerate}
\end{rmk}

\begin{proof}[Proof of Theorem \ref{t:lowdGRStopology}]
To show item (1), we first claim that the fundamental group is finite.  If it were infinite, then a finite cover of $M$ admits a splitting by Theorem \ref{t:splitting}.  Moreover by Theorem \ref{t:splitting} item (3) it follows that the soliton is trivial, a contradiction.  Hence the fundamental group is finite, hence $M$ is a space form by the Poincar\`e Conjecture.  

For item (2) we first show that $\brs{\pi_1(M)} = \infty$.  If not, by Hurewicz we know $b_1 = 0$, hence $b_3 = 0$ by Poincar\'e duality, which is impossible by Proposition \ref{p:b3bound}.  Hence there is a nontrivial torus factor $T^k$ in the cover guaranteed by Theorem \ref{t:splitting}.  If $k = 2$, then by item (3) $H$ is pulled back from a two-dimensional manifold $N$ and hence vanishes, rendering the soliton trivial.  Similarly the cases $k = 3,4$ are ruled out.  Hence only the case $k = 1$ remains, and the transverse data on $N$ is a three-dimensional nontrivial GRS, hence $N \cong S^3 / \Gamma$ by part (1).  Hence $M$ is a finite quotient of $S^3 \times S^1$ as claimed.
\end{proof}

\begin{rmk} There exists a nontrivial family of GRS on $S^3$ (cf. \cite[Corollary 1.4]{Streetssolitons}).  Theorem \ref{t:lowdGRStopology} lends some credence to the conjecture that these are the only such solitons, and that their products with $S^1$ are the only four-dimensional GRS.
\end{rmk}

\section{Topology of string backgrounds}

Here we turn to topological results on string backgrounds, defined generically as Riemannian manifolds equipped with a closed three-form such that the associated Bismut connection has reduced holonomy (cf. \cite{Strominger1986}).  We will focus on the cases of Bismut-Hermitian-Einstein, strong torsion $G_2$, and strong torsion $\Spin(7)$-manifolds.  To keep this paper brief we refer the reader to the recent unified discussion of these geometries \cite{kennon2025canonical} for relevant background.

\subsection{Bismut-Hermitian-Einstein manifolds}

\begin{prop} \label{p:BHEtopology} Let $(M^{2n}, g, J)$ be a nontrivial compact BHE manifold.  Then
\begin{enumerate}
    \item $b_3 \geq 1$,
    \item $Y_{[g]} > 0$,
    \item $\chi = 0$,
    \item $h^{n-1,0} = h^{n,0} = 0$,
    \item $h_T^{n-2,0} = h_T^{n-1,0} = 0$.
\end{enumerate}
\begin{proof}
  As BHE are generalized Ricci solitons \cite[Proposition 2.6]{JFS}, items (1) and (2) follow from Propositions \ref{p:b3bound} and \ref{p:AroofGRS} respectively.  To show item (3) first recall that nontrivial compact BHE admit a canonical Bismut-parallel distribution $\{V, JV \}$ \cite[Proposition 2.6]{JFS}.  Hence there exists a nowhere vanishing vector field, hence $\chi = 0$.  The first claim of item (4) follows from a Bochner argument detailed in \cite[Proposition 8.3]{PCFReg}.  An inspection of the argument shows that it also implies $h^{n,0} = 0$.  For item (5), we claim that the Bochner argument used for item (4) can be repeated for transverse holomorphic forms.  By the computations of \cite[Lemma 6.6]{garcia2026pluriclosed} it follows that the second Chern Ricci curvature of the transverse Hermitian metric satisfies the same equation as the total space BHE metric, up to a bundle curvature term $F^2$ which has a favorable sign.  Hence the arguments for item (4) carry through to prove item (5).
\end{proof}
\end{prop}

\subsection{Strong torsion \texorpdfstring{$G_2$}{G2}-manifolds}

We observe rough constraints on the topology of strong torsion $G_2$ manifolds coming from the discussion of \S \ref{s:topologyGRS}.  Note that item (3) of the proposition below generalizes \cite[Corollary 3.5]{fino2025some}, which shows vanishing of $\Aroof$ under further hypotheses on the torsion.

\begin{prop} \label{p:strongG2topology} Let $(M^7, \varphi)$ be a nontrivial compact strong torsion $G_2$ manifold.  Then
\begin{enumerate}
    \item $b_3 \geq 1$,
    \item $Y_{[g_{\varphi}]} > 0$,
    \item $\Aroof = 0$.
\end{enumerate}
\begin{proof} Since strong torsion $G_2$ manifolds define generalized Ricci solitons \cite{ivanov2023riemannianG2}, the claims (1) and (2) follow from Proposition \ref{p:b3bound} and \ref{p:AroofGRS}.  Since manifolds with a $G_2$ structure are automatically spin, item (3) follows.
\end{proof}
\end{prop}

\subsection{Strong torsion \texorpdfstring{$\Spin(7)$}{Spin(7)}-manifolds}

We turn to the case of strong torsion $\Spin(7)$-manifolds.  As above we establish vanishing of $\Aroof$ which allows us to adapt \cite[Theorem 8.2]{ivanov2004connections} to yield vanishing of several invariants.

\begin{prop} Let $(M^8, \Phi)$ be a nontrivial compact strong torsion $\Spin(7)$-manifold.  Then
\begin{enumerate}
    \item $b_3 \geq 1$,
    \item $Y_{[g_{\Phi}]} > 0$,
    \item $\Aroof = 0$,
    \item $\chi = \tau = p_1^2 = p_2 = 0$.
\end{enumerate}
\begin{proof} Compact strong torsion $\Spin(7)$-manifolds are generalized Ricci solitons \cite{ivanov2023riemannianspin7}, hence items (1) and (2) follow from Propositions \ref{p:b3bound} and \ref{p:AroofGRS}.  As manifolds with $\Spin(7)$ structures are automatically spin, item (3) follows.  Compact nontrivial strong torsion $\Spin(7)$-manifolds admit a canonical Bismut parallel vector field by \cite[Theorem 6.4]{ivanov2023riemannianspin7}.  This vector field is in particular nowhere vanishing, hence $\chi = 0$. 

Using the Hirzebruch signature theorem (cf. \cite{joyce2000compact}) we obtain 
\begin{align*}
    45 (b_4^+ - b_4^-) = 45 \tau = 7 p_2 - p_1^2, \qquad 45 \cdot 2^7 \Aroof = 7 p_1^2 - 4 p_2.
\end{align*}
On the other hand the existence of a $\Spin(7)$ structure always implies
\begin{align*}
    p_1^2 - 4 p_2 + 8 \chi = 0.
\end{align*}
As we have shown that $\chi = \Aroof  = 0$, it follows that $p_1^2 = p_2 = 0$, hence also $\tau = 0$, finishing item (4).
\end{proof}
\end{prop}

\begin{rmk} \label{r:isolatedmoudli} The results above confirm a general principle regarding string backgrounds and their classical special holonomy counterparts, namely that their moduli spaces are distinct, and in many cases cannot even exist on the same manifolds.  From \cite[Proposition 3.5]{Streetssolitons} we know that K\"ahler manifolds cannot support a non-K\"ahler BHE metric, hence the topological backgrounds are disjoint.  Similarly, as manifolds with  $\Spin(7)$-holonomy have $\Aroof = 1$, so again the topological backgrounds are disjoint from strong torsion $\Spin(7)$ manifods.  We cannot rule out the possibility that $G_2$ holonomy metrics can coexist on the same manifold as a strong torsion $G_2$ structure.  However, the stability result \cite[Theorem 1.2]{dai2005stability} says in particular that special holonomy metrics on a simply connected manifold cannot admit a nearby metric with positive scalar curvature.  It hence follows that on a given manifold the moduli space of $G_2$ holonomy metrics is disjoint from the moduli of strong torsion $G_2$ structures.
\end{rmk}

\section{Diffeomorphism Classification of BHE threefolds}

Here we turn to a more refined discussion of the topology of BHE threefolds, leading to the proofs of Theorems \ref{t:BHE3foldssmooth}, \ref{t:BHE3foldskollar}, and \ref{t:BHE3foldsrank0}.  We note that all homology and cohomology spaces below are with integer coefficients unless otherwise specified.

\subsection{Topological lemmas}

\begin{lemma} \label{l:spin} Let $(M^{2n},g, J)$ be a compact BHE with $\pi_1(M)$ free abelian.  Then $M$ is spin.

\begin{proof}
    Since $\pi_1(M)$ is free abelian it follows that $H_1(M)$ is torsion-free, thus the universal coefficient theorem gives $H^2(M) \cong
\operatorname{Hom}\!\left(H_2(M),\mathbb Z\right)$, hence $H^2(M)$ is torsion-free.  Since $M$ is BHE we have $c_1(M; \mathbb R) = 0$, hence $c_1(M; \mathbb Z)=0$.  Since $M$ is complex, $w_2(M)\equiv c_1(M)$ mod $2$, thus $w_2(M) = 0$.
\end{proof}
\end{lemma}

\begin{lemma}\label{lem:torsionfree}
Let $M$ be a compact $6$-manifold such that $H_1(M)$ and $H_2(M)$ are
torsion-free. Then $H_*(M)$ is torsion-free.
\end{lemma}

\begin{proof}
The universal coefficient theorem implies that
\[
\operatorname{Tor}H^k(M)
\cong
\operatorname{Ext}\!\left(H_{k-1}(M),\mathbb Z\right)
\cong
\operatorname{Tor}H_{k-1}(M).
\]
Taking $k=6-i$ and using Poincar\'e duality we obtain
\[
\operatorname{Tor}H_i(M)
\cong
\operatorname{Tor}H^{6-i}(M)
\cong
\operatorname{Tor}H_{5-i}(M).
\]
Thus since $H_1(M)$ and $H_2(M)$ are torsion-free, so are $H_4(M)$ and $H_3(M)$.  Furthermore, Poincar\'e duality and the universal coefficient theorem give
\[
H_5(M)
\cong
H^1(M)
\cong
\operatorname{Hom}\!\left(H_1(M),\mathbb Z\right),
\]
so $H_5(M)$ is free abelian.  Since of course $H_0(M)\cong H_6(M)\cong \mathbb Z$, the lemma follows.
\end{proof}

\begin{lemma}\label{lem:p1-vanishing}
Let \(M\) be a compact smooth oriented six-manifold equipped with a locally
free \(T^2\)-action, and suppose one of the two circle Euler
classes is nonzero.  Then $p_1(TM; \R) = 0$.
\end{lemma}

\begin{proof}
Let \(\mathcal V\subset TM\) be the vertical tangent bundle.  The two
infinitesimal generators of the \(T^2\)-action give a global framing of
\(\mathcal V\), which is hence trivial.  After choosing a \(T^2\)-invariant horizontal distribution \(\mathcal H\),
we have
\[
    TM=\mathcal V\oplus\mathcal H,
    \qquad
    \mathcal H\cong \pi^*TB,
\]
where \(TB\) denotes the orbifold tangent bundle of \(B\).  Therefore the
Whitney product formula gives
\begin{equation}\label{eq:p1-pullback}
    p_1(TM; \mathbb R)=\pi^*p_1^{\mathrm{orb}}(TB; \mathbb R).
\end{equation}

We now show that the pullback of every class in
\(H^4_{\mathrm{orb}}(B;\R)\) vanishes.  Since \(B\) is a compact connected
oriented four-orbifold, orbifold Poincar\'e duality gives a nondegenerate
pairing
\[
    H^2_{\mathrm{orb}}(B;\R)\times H^2_{\mathrm{orb}}(B;\R)
    \longrightarrow H^4_{\mathrm{orb}}(B;\R)\cong\R.
\]
Because \(e\neq 0\), there exists \(b\in H^2_{\mathrm{orb}}(B;\R)\) such
that
\[
    e\smile b\neq 0.
\]
Since \(H^4_{\mathrm{orb}}(B;\R)\) is one-dimensional, \(e\smile b\)
spans it.  Hence every \(a\in H^4_{\mathrm{orb}}(B;\R)\) is expressed $a=c\,e\smile b$ for some \(c\in\R\).

Let \(\eta\) be a connection one-form corresponding to the chosen circle
factor, and let \(F\) be its orbifold curvature form on \(B\).  Thus
\[
    d\eta=\pi^*F,
    \qquad
    [F]=e.
\]
Choose a closed orbifold two-form \(\beta\) representing \(b\).  Then
\[
\begin{aligned}
    \pi^*a
    &=c\,[\pi^*(F\wedge\beta)]\\
    &=c\,[d\eta\wedge\pi^*\beta]\\
    &=c\,[d(\eta\wedge\pi^*\beta)]\\
    &=0.
\end{aligned}
\]
Thus
\[
    \pi^*\colon H^4_{\mathrm{orb}}(B;\R)\longrightarrow H^4(M;\R)
\]
is the zero map, finishing the proof.
\end{proof}

\subsection{Regular case}

\begin{lemma}\label{lem:S1bundle}
Let $N$ be a closed, simply connected, spin $5$-manifold which is a principal $S^1$ bundle over $B^4$.  Assuming $H_2(B)\cong \mathbb Z^r$, then $r\geq 1$ and
\[
N\cong \#^{\,r-1}(S^2\times S^3).
\]
\end{lemma}

\begin{proof}
The homotopy long exact sequence implies that $\pi_1(B)=0$.  Let $
e\in H^2(B)$ 
denote the Euler class of the circle bundle. The homotopy exact
sequence also gives
\[
0\longrightarrow \pi_2(N)
\longrightarrow \pi_2(B)
\xrightarrow{\delta}
\pi_1(S^1)\cong\mathbb Z
\longrightarrow 0.
\]
Using the Hurewicz identification $
\pi_2(B)\cong H_2(B)$, 
the boundary map is evaluation against the Euler class:
\[
\delta(A)=\langle e,A\rangle.
\]
In particular, $\delta$ is surjective, so $e$ is primitive and
$r\geq 1$. Its kernel is therefore free abelian of rank $r-1$.  Since $N$ is simply connected, the Hurewicz theorem gives
\[
H_2(N)\cong \pi_2(N)\cong\ker\delta\cong\mathbb Z^{r-1}.
\]
Thus $N$ is a closed, simply connected, spin $5$-manifold with
torsion-free second homology of rank $r-1$.  The lemma follows from the Smale--Barden
classification \cite{barden1965simply,smale1962structure}.
\end{proof}

\begin{lemma}\label{lem:T2bundle}
Let $M$ be a compact simply connected spin $6$-manifold which is a principal $T^2$ bundle over $B^4$.  Assuming 
$
H_2(B)\cong \mathbb Z^r$, then $r\geq 2$ and
\[
M\cong
\#^{\,r-2}(S^2\times S^4)
\#^{\,r-1}(S^3\times S^3).
\]
\end{lemma}

\begin{proof}
The homotopy exact sequence of the bundle contains
\[
\pi_1(M)\longrightarrow \pi_1(B)\longrightarrow \pi_0(T^2).
\]
Since $M$ is simply connected and $T^2$ is connected, it follows that
$\pi_1(B)=0$.  Choose an identification $T^2=S^1\times S^1$ and denote the Euler classes by $e_1,e_2\in H^2(B)$.  The homotopy exact sequence also gives
\[
0\longrightarrow \pi_2(M)
\longrightarrow \pi_2(B)
\xrightarrow{\delta}
\pi_1(T^2)\cong\mathbb Z^2
\longrightarrow 0,
\]
where, using the Hurewicz identification $\pi_2(B)\cong H_2(B)$, 
the boundary map is
\[
\delta(A)
=
\bigl(
\langle e_1,A\rangle,
\langle e_2,A\rangle
\bigr).
\]
In particular, $\delta$ is surjective. Hence $r\geq 2$, and its kernel
is free abelian of rank $r-2$. Since $M$ is simply connected, the
Hurewicz theorem gives
\[
H_2(M)\cong \pi_2(M)\cong\ker\delta\cong\mathbb Z^{r-2}.
\]
By Lemma \ref{lem:torsionfree}, the homology of $M$ is torsion-free.

By Poincar\'e duality we have $b_4(M)=b_2(M)=r-2$.  Since further $\chi(M) = 0$ and using $b_1(M)=b_5(M)=0$, we obtain
\[
0
=
\chi(M)
=
2+2b_2(M)-b_3(M),
\]
and therefore $b_3(M)=2r-2$.

We next show that the cup product vanishes on $H^2(M)$. Dualizing the short exact sequence above and using the universal coefficient theorem, we obtain
\[
0 \longrightarrow \mathbb{Z}^2
\xrightarrow{\delta^*} H^2(B)
\xrightarrow{\pi^*} H^2(M)
\longrightarrow 0.
\]
Since $\delta^*(1,0)=e_1$ and $\delta^*(0,1)=e_2$, it follows that
\[
H^2(M) \cong H^2(B)/\langle e_1,e_2\rangle.
\]
In particular, every class in $H^2(M)$ is the pullback of a class in $H^2(B)$.

The surjectivity of $\delta$ implies that $e_1$ is primitive. Since
the intersection form of the closed simply connected $4$-manifold
$B$ is unimodular, there exists $\alpha\in H^2(B)$ such that
\[
\langle e_1\smile\alpha,[B]\rangle=1.
\]
Thus $e_1\smile\alpha$ generates $H^4(B)$. Since 
$\pi^*e_1=0,$ we have $\pi^*(e_1\smile\alpha)=0$, thus $
\pi^*\colon H^4(B)\longrightarrow H^4(M)$
is the zero map. Consequently, for all $\beta,\gamma\in H^2(B)$,
\[
\pi^*\beta\smile\pi^*\gamma
=
\pi^*(\beta\smile\gamma)
=
0.
\]
Thus the cup product vanishes identically on $H^2(M)$.

We have shown that $M$ is simply connected and spin, has torsion-free
homology,
\[
H_2(M)\cong\mathbb Z^{r-2},
\qquad
H_3(M)\cong\mathbb Z^{2r-2},
\]
and has trivial cup product on $H^2(M)$.  It further follows from Lemma \ref{lem:p1-vanishing} that $p_1(M)=0$.
By the Wall--Jupp classification \cite{jupp1973classification,wall1966classification}, these invariants determine $M$ up
to diffeomorphism and agree with those of $
\#^{\,r-2}(S^2\times S^4)
\#^{\,r-1}(S^3\times S^3)$, finishing the proof.
\end{proof}

\begin{lemma} \label{l:rank1structure} Let $(M, g, J)$ be a compact BHE threefold such that the finite cover guaranteed by Theorem \ref{t:splitting} has $k = 1$.  Then $V$ generates the split $S^1$ factor.
\begin{proof} Let $Z$ denote the canonical vector field for the isometric $S^1$ action determined by the splitting on the canonical finite cover.  Consider the splitting $Z = Z^{\mathcal V} + Z^{\mathcal H}$, and first assume that $Z^{\mathcal H} \neq 0$.  Noting that $\N Z^{\mathcal H} \equiv \N J Z^{\mathcal H} \equiv 0$, it follows that the Bismut holonomy must be trivial, in which case it follows from the classification of Bismut flat manifolds that the canonical finite cover is isometric to a product of a Bismut flat Hopf surface and a flat torus, hence in fact $k = 3$.  Thus $Z^\mathcal H = 0$.  Thus $Z \in \mathrm{span} \{V, JV \}$, and using that it generates a trivial $S^1$ action, it follows from the Hermitian-Einstein equation for $JV$ that $\IP{Z, JV} = 0$.  Thus $Z$ is a multiple of $V$, as claimed.
\end{proof}
\end{lemma}

\begin{proof}[Proof of Theorem \ref{t:BHE3foldssmooth}] 
As BHE are generalized Ricci solitons, without loss of generality we consider the cover guaranteed by Theorem \ref{t:splitting}.  In case $k \geq 2$ it follows from Theorem \ref{t:lowdGRStopology} that $N \cong S^3 \times S^1$, so that in fact $k = 3$ and $N \cong S^3$ as claimed.

In case $k \leq 1$, due to the assumption that the canonical action is principal, we obtain a smooth complex surface as its quotient.  In \cite[Lemma 2.3]{apostolov2026pluriclosed} a characterization of the possibilities was shown using the Kodaira classification.  One possibility is a minimal ruled surface over an elliptic curve, however these have $h^{1,0} = 1$, and are thus ruled out by Proposition \ref{p:BHEtopology}.  The remaining possibilities are $\mathbb P^1 \times \mathbb P^1$ or a Hirzebruch surface blown up at at most $8$ points.  Moreover, it is shown in \cite[Lemma 2.3]{apostolov2026pluriclosed} that $c_1(B)^2 \geq 0$ with equality if and only if $V$ generates a product action.  

Let us now suppose that $k = 1$.  It follows from Lemma \ref{l:rank1structure} that $V$ generates a product action, hence $c_1(B)^2 = 0$, and it follows that $B$ is a blowup of either $\mathbb P^1 \times \mathbb P^1$ or a Hirzebruch surface at $8$ points.  In particular $b_2(B) = 10$, and the claim follows from Lemma \ref{lem:S1bundle}.

In case $k = 0$, $V$ cannot generate a product action, thus $c_1^2(B) > 0$.  Hence $B$ is a blowup of either $\mathbb P^1 \times \mathbb P^1$ or a Hirzebruch surface at between $0$ and $7$ points.  In particular $2 \leq b_2(B) \leq 9$, and the claim follows from Lemma \ref{lem:T2bundle}.
\end{proof}

\subsection{Quasiregular case}

We turn now to the quasiregular case, and the proofs of Theorems \ref{t:BHE3foldskollar} and \ref{t:BHE3foldsrank0}.  It turns out we already have the necessary tools to prove Theorem \ref{t:BHE3foldskollar}, after a further key input from \cite{kollar2006circle}:

\begin{proof}[Proof of Theorem \ref{t:BHE3foldskollar}] Since $\brs{\pi_1(M)} = \infty$, the canonical cover $N \times T^k$ guaranteed by Theorem \ref{t:splitting} satisfies $k \geq 1$.  If $k \geq 2$ then Theorem \ref{t:lowdGRStopology} implies that in fact $k = 3$ and $N \cong S^3$, yielding item (2).  Thus we now assume $k = 1$.  By Lemma \ref{l:rank1structure}, we know that $V$ generates the split $S^1$ direction, so that $N$ is a simply connected Seifert fibration over a complex orbifold $(B, \Delta)$, where
\[
\Delta
=
\sum_i
\Bigl(1-\frac1{m_i}\Bigr)D_i.
\]
By \cite[Proposition 28]{kollar2006circle},
\[
\operatorname{Tor} H_2(M)
\cong
\bigoplus_i (\mathbb Z/m_i)^{2g(D_i)}.
\]
Since every orbifold divisor \(D_i\) has genus zero, each exponent
\(2g(D_i)\) is zero, hence $H_2(N)$ is torsion-free.  It further follows from Lemma \ref{l:spin} that $N$ is spin, hence the result follows from the Smale-Barden classification \cite{barden1965simply,smale1962structure}.
\end{proof}

In view of this proof and Lemma \ref{lem:torsionfree} the main task in proving Theorem \ref{t:BHE3foldsrank0} is then to provide a generalization of Kollar's torsion computation to the more general setting of a locally free $T^2$ action with complex orbifold quotient.  In full generality such a characterization of the torsion seems challenging, in part due to the fact that the stabilizer subgroups need not always be cyclic, a fact which is automatic in the Seifert fibration case and exploited in \cite{kollar2006circle}.  We will show that for BHE threefolds the stabilizers at generic divisorial points are always cyclic, and cyclic at intersection points under a natural coprimality condition.  This allows for a decisive extension of \cite[Proposition 28]{kollar2006circle}.

\subsubsection{Analysis of stabilizers}

Here we analyze the stabilizer subgroups for $T^2$ actions on six-manifolds.  We first fix notation and  basic observations used throughout this discussion.  For notational ease let $T = T^2$, and let 
\begin{align*}
    \Gamma_x = \operatorname{Stab}_T(x),
\end{align*}
By assumption $\Gamma_x$ is a finite abelian group.  The slice theorem yields the existence of a neighborhood of the orbit $T \cdot x$ equivariantly diffeomorphic to $T \times_{\Gamma_x} V$, where $V$ is the slice representation.  As our actions are isometric for a Riemannian metric and effective, the representation of $\Gamma_x$ on $V$ is always faithful.  Moreover, as the action is holomorphic, $V \cong \mathbb C^2$ and $\Gamma_x$ acts via elements of $\mathrm{U}(2)$.
The orbifold locus thus consists of divisors $D_i$ and points $p_i$.  Our first main goal is to give a simple criteria to show that the resulting orbifold is cyclic.

\begin{lemma}
\label{lem:codim-two-isotropy-cyclic}
If $
\operatorname{codim}_{\mathbb R}V^{\Gamma_x}=2$, then \(\Gamma_x\) is cyclic.
\end{lemma}

\begin{proof}
By hypothesis we obtain an induced faithful representation
\[
\rho:\Gamma_x\longrightarrow SO((V^{\Gamma_x})^{\perp}) \cong SO(2) \cong S^1.
\]
Since finite subgroups of $S^1$ are cyclic, the lemma follows.
\end{proof}

\begin{lemma}
\label{lem:local-divisor-strata} We have a splitting
\[
V=L_{1}\oplus L_{2},
\]
such that every point in the slice chart with nontrivial stabilizer lies
in \(L_{1}\cup L_{2}\). Consequently, at most two divisors may meet at a given point.
\end{lemma}

\begin{proof}
Since \(\Gamma_x\) is finite abelian, the slice representation splits into
characters:
\[
V=L_{1}\oplus L_{2},\qquad
\gamma\cdot(z_{1},z_{2})
=
\bigl(\chi_{1}(\gamma)z_{1},\chi_{2}(\gamma)z_{2}\bigr).
\]
Faithfulness gives
\[
\ker\chi_{1}\cap\ker\chi_{2}=\{1\}.
\]
Thus, if \(z_{1}z_{2}\neq0\), the stabilizer of \((z_{1},z_{2})\) is trivial.  Hence the
only positive-dimensional nonprincipal strata are the lines \(L_{i}\) for
which \(\ker\chi_{i}\neq\{1\}\).  As their images in \(V/\Gamma_x\) are precisely the local codimension-two
orbifold strata, the lemma follows.
\end{proof}

\begin{lemma} \label{l:isolatedcyclicity} Given $x \in M$ such that $\pi(x)$ is an isolated orbifold point, $
\Gamma_x$ is cyclic.
\end{lemma}

\begin{proof}
Using the notation and setup of Lemma \ref{lem:local-divisor-strata}, we claim that each representation $\chi_i$ is faithful. Indeed, if
$\ker\chi_1$ were nontrivial, then every point of the complex line
$L_1\setminus\{0\}$ would have nontrivial stabilizer, and hence determine a codimension-two orbifold stratum
through $p$.  Thus $\chi_1$ gives an embedding $\Gamma_x\hookrightarrow S^1$, 
thus $\Gamma_x$ is cyclic.
\end{proof}

Finally we record a general criterion for cyclicity of the stabilizer at a
transverse intersection of two codimension-two isotropy strata.

\begin{lemma}
\label{lem:cyclic-intersection-stabilizer}
Suppose that divisors 
\(D_i\) and \(D_j\) meet transversely at $\pi(x)$.  Let $\Gamma_i\cong\mathbb Z/m_i,
    \Gamma_j\cong\mathbb Z/m_j$
be the generic stabilizers along \(D_i\) and \(D_j\), respectively.  Then
\[
    \Gamma_x \text{ is cyclic}
    \quad\Longleftrightarrow\quad
    \gcd(m_i,m_j)=1.
\]
\end{lemma}
\begin{proof}
Using Lemma \ref{lem:local-divisor-strata} we obtain a splitting $V = L_1 \oplus L_2$ and a splitting of the slice representation into characters $\chi_1, \chi_2$.  We relabel the divisors $D_1, D_2$ to correspond to the image of the two planes $L_1$ and $L_2$, so that then $\Gamma_1=\ker \chi_1, \Gamma_2=\ker \chi_2$.
Since the slice representation is faithful,
\[
    \Gamma_i\cap\Gamma_j
    =
    \ker\chi_1\cap\ker\chi_2
    =
    \ker(\chi_1,\chi_2)
    =
    \{1\}.
\]
First assuming \(\Gamma_x\) is cyclic, we know that for two subgroups of a finite
cyclic group, the order of their intersection is the greatest common
divisor of their orders, i.e.
\[
    \gcd(m_i,m_j) = |\Gamma_i\cap\Gamma_j| = 1,
\]
as claimed.

Now suppose that $\gcd(m_1,m_2)=1$, and 
assume \(\Gamma_x\) is not cyclic.  Since
\(\Gamma_x\) is finite abelian there exists a prime \(p\) for which the
\(p\)-primary subgroup $(\Gamma_x)_{(p)}$ is not cyclic.
Note that the image $\chi_1( (\Gamma_x)_{(p)})$ is a finite subgroup of $S^1$, hence cyclic.  It follows that $(\Gamma_x)_{(p)}$ has nontrivial intersection with $\ker \chi_1$, hence $p$ divides $m_1$.  The same argument shows that also $p$ divides $m_2$, contradicting $\gcd(m_1, m_2) = 1$.
\end{proof}

\begin{rmk} Note that in the setup of Lemma \ref{lem:cyclic-intersection-stabilizer}, the stabilizer $\Gamma_x$ need not be generated by $\Gamma_i$ and $\Gamma_j$.  Indeed, the intersection point is generically an orbifold point with group $\Gamma_x / \IP{\Gamma_i, \Gamma_j}$.
\end{rmk}

\subsubsection{Torsion Computation}

Here we prove the generalization of \cite[Proposition 28]{kollar2006circle}.  We first give the statement then prove a series of lemmas leading to the proof.

\begin{prop} \label{p:torsion-formula}
Let $M^6$ be a compact simply connected manifold with a locally
free $T^2$-action with orbit space $B$.  Assume:
\begin{enumerate}
    \item $B$ is cyclic,
    \item The orbifold locus of $B$ consists of a finite union of points and a union of closed oriented surfaces 
    \[
        D=\bigcup_{i=1}^r D_i
    \]
    meeting transversely,
    \item for every intersection point $p\in D_i\cap D_j$, the natural map
    \[
\Gamma_i\oplus\Gamma_j\longrightarrow\Gamma_p
    \]
    is injective, where $\Gamma_i\cong\mathbb Z/m_i$ is the generic
    stabilizer along $D_i$.
\end{enumerate}
Then
\[
    \Tor H_2(M)
    \cong
    \bigoplus_{i=1}^r H_1(D_i;\mathbb Z/m_i)
    \cong
    \bigoplus_{i=1}^r (\mathbb Z/m_i)^{2g(D_i)}.
\]
\end{prop}

The main argument works via the associated Borel construction.  To fix notation, first let $E T^k$ denote the universal free $T^k$-space, and let $B T^k = E T^k / T^k$ denote the classifying space.  We then denote the Borel construction via
\begin{align*}
    \bar{B} := ET^2 \times_{T^2} M.
\end{align*}
We first show that $\bar{B}$ is simply connected, and the second homology groups of $M$ and $\bar{B}$ are related in simple fashion:
\begin{lemma}\label{lem:borel-to-total-space}
Given $M$ a compact simply connected manifold with a $T^k$ action, then $\pi_1(\bar{B}) = 0$, and there is a short exact sequence
\[
    0\longrightarrow H_2(M)
    \longrightarrow H_2(\bar B)
    \longrightarrow \mathbb Z^k
    \longrightarrow 0.
\]
In particular,
\[
    \Tor H_2(M)
    \cong
    \Tor H_2(\bar B).
\]
\end{lemma}

\begin{proof}
As $\bar{B}$ is an $M$-fiber bundle over $BT^k$, from the homotopy long exact sequence we obtain
\[
    \pi_1(M) \longrightarrow \pi_1(\bar{B}) \to \pi_1(B T^k).
\]
Since $\pi_1(M) = \pi_1(B T^k) = 0$ it follows that $\pi_1(\bar{B}) = 0$.  Also from the long exact sequence we obtain
\[
    \pi_3(BT^k)\longrightarrow \pi_2(M)
    \longrightarrow \pi_2(\bar B)
    \longrightarrow \pi_2(BT^k)
    \longrightarrow \pi_1(M).
\]
Since
\[
    \pi_3(BT^k)= \pi_2(T^k) = 0,
    \qquad
    \pi_2(BT^k) = \pi_1(T^k) \cong\mathbb Z^k,
    \qquad
    \pi_1(M)=0,
\]
we obtain a short exact sequence
\[
    0\longrightarrow \pi_2(M)
    \longrightarrow \pi_2(\bar B)
    \longrightarrow \mathbb Z^k
    \longrightarrow 0.
\]
Using simply connectivity of $M$ and $\bar{B}$, the Hurewicz theorem finishes the proof.
\end{proof}

We approach the homology of $\bar{B}$ via a spectral sequence argument.  Let $B=M/T^2$ denote the coarse quotient and let
\[
\bar B=ET^2\times_{T^2}M
\]
denote the Borel construction. Choose a finite CW decomposition of $B$ adapted to the orbit-type stratification: each orbifold divisor $D_i$ is a subcomplex, every divisor intersection and isolated orbifold point is a vertex, and the stabilizer is constant over the interior of each cell.  For a cell $\sigma\subset B$, denote this stabilizer by $\Gamma_\sigma$.  Over the interior of a cell $\sigma$, the Borel construction is homotopy equivalent to
$B\Gamma_\sigma\times\mathring{\sigma}$. 
More explicitly, a $p$-cell $\sigma\subset B$ with stabilizer
$\Gamma_\sigma$ contributes to the Borel construction a relative piece
modeled on
\[
B\Gamma_\sigma\times (D^p,S^{p-1}).
\]
After choosing a CW model for $B\Gamma_\sigma$, each $q$-cell of
$B\Gamma_\sigma$ therefore gives a cell of total dimension $p+q$ lying
over $\sigma$.  We filter these cells by the dimension $p$ of the
underlying cell of the coarse quotient $B$.  Thus the associated
bigraded cellular complex has, in bidegree $(p,\ast)$, a summand
corresponding to the cellular chains of $B\Gamma_\sigma$ for each
$p$-cell $\sigma\subset B$.  The spectral sequence computes the homology of the
Borel construction $\bar B$, while the homology of the coarse quotient
$B$ appears naturally in its bottom row, and the higher rows record the additional homology arising from the
stabilizer groups.

\begin{lemma} \label{lem:spectralpage2} One has
\[
    E^{1}_{p,q}
    =
    \bigoplus_{\dim \sigma=p} H_{q}(B\Gamma_{\sigma}).
\]
If every stabilizer is finite cyclic, then one has 
\[
    E^{2}_{2,0}\cong H_{2}(B),
    \qquad
    E^{2}_{1,1}\cong H_{1}(C^{\mathrm{stab}}_{*}),
    \qquad
    E^{2}_{0,2}=0,
\]
where
\[
    C^{\mathrm{stab}}_{p}
    :=
    \bigoplus_{\dim \sigma=p}\Gamma_{\sigma}.
\]
\end{lemma}

\begin{proof}
Over the interior of a \(p\)-cell \(\sigma\), the Borel construction is
homotopy equivalent to \(B\Gamma_{\sigma}\times \mathring{D}^{p}\).
Excision therefore identifies the relative homology of successive filtered
pieces with \(H_{q}(B\Gamma_{\sigma})\), giving the stated
\(E^{1}\)-page.  For the second page, recall that for a finite cyclic group \(\Gamma\),
\[
    H_{q}(B\Gamma)
    \cong
    \begin{cases}
        \mathbb{Z}, & q=0,\\
        \Gamma, & q>0 \text{ odd},\\
        0, & q>0 \text{ even}.
    \end{cases}
\]
Hence the \(q=0\) row is the ordinary cellular chain complex of \(B\), while
the \(q=1\) row is \(C^{\mathrm{stab}}_{*}\). The claimed identifications
follow.
\end{proof}

\begin{lemma}\label{lem:stabilizer-homology}
There is a natural isomorphism
\[
    H_1(C_*^{\mathrm{stab}})
    \cong
    \bigoplus_{i=1}^r H_1(D_i; \mathbb Z/m_i).
\]
\end{lemma}

\begin{proof}
Let
\[
    A_*:=\bigoplus_{i=1}^r C_*(D_i;\Gamma_i).
\]
The inclusions of the generic divisor stabilizers define a chain map
$A_*\to C_*^{\mathrm{stab}}$. Away from the zero-cells this map is an
isomorphism. At a divisor intersection, hypothesis \textup{(5)} makes the map
from the sum of the incident generic stabilizers injective. It follows that
there is a short exact sequence of chain complexes
\[
    0\longrightarrow A_*
    \longrightarrow C_*^{\mathrm{stab}}
    \longrightarrow Q_*
    \longrightarrow 0,
\]
where $Q_*$ is concentrated in degree zero. Hence
\[
    H_1(C_*^{\mathrm{stab}})
    \cong H_1(A_*)
    \cong \bigoplus_{i=1}^r H_1(D_i;\Gamma_i),
\]
as claimed.
\end{proof}

Lemma \ref{lem:stabilizer-homology} contains the main geometric content of Proposition \ref{p:torsion-formula}.  However we must further analyze the spectral sequence to show that no other terms can contribute torsion.  The main idea is to show that the orbifold points cannot contribute torsion by removing neighborhoods around them and relative cohomology and excision.  We first use this idea to derive relevant homology groups on the coarse space.  Let $P$ be the finite set of genuine coarse singular points of $B$. Choose disjoint
closed conical neighborhoods $V_p\cong cL_p$, where $L_p$ is the relevant lens space at $p$, and set
\[
    W:=B\setminus\bigcup_{p\in P}\Int(V_p).
\]
Then $W$ is a compact connected oriented four-manifold with
\[
    \partial W=\coprod_{p\in P}L_p.
\]

\begin{lemma}
\label{lem:topology-of-coarse-space}
One has
\[
    H_3(B)=0, \qquad \Tor H_2(W) = 0
\]
\end{lemma}

\begin{proof}
Since \(B\) is obtained from \(W\) by attaching the contractible cones
\(cL_p\), van Kampen's theorem gives
\[
    \pi_1(B)
    \cong
    \pi_1(W)\Big/ \left\langle\operatorname{im}\pi_1(\partial W)
    \right\rangle.
\]
Since \(B\) is simply connected, \(\pi_1(W)\) is normally generated by
the images of the groups \(\pi_1(L_p)\). Passing to abelianizations,
we conclude that
\[
    H_1(\partial W)
    \longrightarrow
    H_1(W)
\]
is surjective.

The long exact homology sequence of the pair \((W,\partial W)\)
therefore gives an isomorphism
\[
    H_1(W,\partial W)
    \cong
    \ker\!\left(
        H_0(\partial W)
        \longrightarrow
        H_0(W)
    \right).
\]
Since \(W\) is connected and \(\partial W\) has \(|P|\) connected
components, the map on \(H_0\) is the summation map from $\mathbb Z^{|P|}$ to $
    \mathbb Z$,
thus
\[
    H_1(W,\partial W)
    \cong
    \mathbb Z^{|P|-1}.
\]
In particular, \(H_1(W,\partial W)\) is free abelian.
The universal coefficient theorem then implies that
\(H^2(W,\partial W)\) is free abelian. By
Poincar\'e--Lefschetz duality,
\[
    H_2(W)
    \cong
    H^2(W,\partial W),
\]
so \(H_2(W)\) is free abelian.

We next prove that \(H_3(B)=0\). The long exact cohomology
sequence of \((W,\partial W)\), together with the injectivity of
\[
    H^1(W)
    \longrightarrow
    H^1(\partial W),
\]
gives
\[
    H^1(W,\partial W)
    \cong
    \operatorname{coker}\!\left(
        H^0(W)
        \longrightarrow
        H^0(\partial W)
    \right)
    \cong
    \mathbb Z^{|P|-1}.
\]
Thus Poincar\'e--Lefschetz duality gives $H_3(W)
    \cong
    \mathbb Z^{|P|-1}$.
By excision and the contractibility of each cone \(cL_p\),
\[
    H_4(B,W)
    \cong
    \bigoplus_{p\in P}H_4(cL_p,L_p)
    \cong
    \bigoplus_{p\in P}H_3(L_p)
    \cong
    \mathbb Z^{|P|},
\]
whereas
\[
    H_3(B,W)
    \cong
    \bigoplus_{p\in P}H_2(L_p)
    =0.
\]
The boundary homomorphism
\[
    H_4(B,W)
    \longrightarrow
    H_3(W)
\]
maps the relative fundamental class of \(V_p\) to the fundamental
class of the corresponding boundary component \(L_p\). These boundary
classes generate \(H_3(W)\), with the unique relation that
their sum is zero. Hence the boundary homomorphism is surjective.
The long exact sequence of the pair \((B,W)\) hence yields $H_3(B)=0$, as claimed.
\end{proof}

\begin{lemma}
\label{lem:spectral-sequence-stabilization}
One has
\[
    E^\infty_{1,1}\cong E^2_{1,1}, \qquad \Tor E^{\infty}_{2,0} = 0.
\]
\end{lemma}

\begin{proof}
The only possible differential entering \(E^2_{1,1}\) is $d_2:E^2_{3,0}\longrightarrow E^2_{1,1}$.  By Lemma~\ref{lem:topology-of-coarse-space}, $E^2_{3,0}\cong H_3(B)=0$.  As no differential leaves \(E^2_{1,1}\), and no differential
\(d_r\), with \(r\geq 3\), can enter it, it follows that $E^\infty_{1,1}\cong E^2_{1,1}$.

We next prove that \(E^\infty_{2,0}\) is free. Let $q:\bar B\longrightarrow B$
be the Borel projection, and set $\bar W:=q^{-1}(W), \bar{V}_p := q^{-1} (V_p)$.  By excision,
\[
    H_k(\bar B,\bar W)
    \cong
    \bigoplus_{p\in P}
    H_k(\bar V_p,\partial\bar V_p).
\]
Near the orbit lying over \(p\), the \(T^2\)-action has local model $T^2\times_{\Gamma_p}D^4$, thus 
\[
    (\bar V_p,\partial\bar V_p)
    \simeq
    \left(
        E\Gamma_p\times_{\Gamma_p}D^4,\,
        E\Gamma_p\times_{\Gamma_p}S^3
    \right).
\]
This is the Thom pair of the oriented rank-four vector bundle
\[
    E\Gamma_p\times_{\Gamma_p}\mathbb R^4
    \longrightarrow
    B\Gamma_p.
\]
The Thom isomorphism therefore gives $
    H_k(\bar V_p,\partial\bar V_p)
    \cong
    H_{k-4}(B\Gamma_p)$, hence
\[
    H_2(\bar B,\bar W)
    =
    H_3(\bar B,\bar W)
    =0.
\]
Thus the long exact sequence of the pair gives an isomorphism
\[
    H_2(\bar W)
    \xrightarrow{\;\cong\;}
    H_2(\bar B).
\]
We further require the corresponding statement for the filtrations defining
the spectral sequences.  Without loss of generality we can assume that the CW structure is 
compatible with the inclusion $\bar W\subset\bar B$
and with the cellular filtrations used to construct these spectral
sequences. For each \(p\in P\), the Thom pair
\[
    \left(
        E\Gamma_p\times_{\Gamma_p}D^4,\,
        E\Gamma_p\times_{\Gamma_p}S^3
    \right)
\]
admits a relative CW structure whose cells are obtained from the cells
of \(B\Gamma_p\) by increasing their dimensions by four. Hence the
relative pair \((\bar B,\bar W)\) has no relative cells in dimensions at most
three.

It follows that the inclusion of filtered cellular chain complexes
\[
    C_*(\bar W)
    \longrightarrow
    C_*(\bar B)
\]
is an isomorphism in chain degrees at most three. Since the filtration
on second homology is determined by the filtered portion
\[
    C_3\longrightarrow C_2\longrightarrow C_1,
\]
the inclusion identifies the induced filtrations on
\(H_2(\bar W)\) and \(H_2(\bar B)\). Passing to
associated graded groups therefore gives $E^\infty_{2,0}(\bar W)
    \cong
    E^\infty_{2,0}(\bar B)$.

It remains to consider the spectral sequence for \(\bar W\). Its
\((2,0)\)-term is
\[
    E^2_{2,0}(\bar W)
    \cong
    H_2(W).
\]
Because the spectral sequence is concentrated in the first quadrant,
there are no incoming differentials to \(E^r_{2,0}(\bar W)\).  For \(r\geq 3\), the target of \(d_r\) has negative first index, hence the only
possible outgoing differential is
\[
    d_2:
    E^2_{2,0}(\bar W)
    \longrightarrow
    E^2_{0,1}(\bar W);
\]
Consequently,
\[
    E^\infty_{2,0}(\bar W)
    =
    E^3_{2,0}(\bar W)
    =
    \ker\!\left(
        d_2:
        H_2(W)
        \longrightarrow
        E^2_{0,1}(\bar W)
    \right).
\]
Thus \(E^\infty_{2,0}(\bar W)\) is a subgroup of
\(H_2(W)\). By
Lemma~\ref{lem:topology-of-coarse-space}, the latter group is free
abelian. Hence \(E^\infty_{2,0}(\bar W)\) is free abelian, and the lemma follows.
\end{proof}

\begin{proof}[Proof of Proposition~\ref{p:torsion-formula}] Because $E_{0,2}^{\infty} = 0$, the filtration of $H_2(\bar{B})$ associated to the spectral sequence gives a short exact
sequence
\[
    0\longrightarrow E^\infty_{1,1}
    \longrightarrow H_2(\bar{B})
    \longrightarrow E^\infty_{2,0}
    \longrightarrow 0,
\]
By Lemma~\ref{lem:spectral-sequence-stabilization}, $E^{\infty}_{2,0}$ is free, so this sequence splits to yield 
\[
    \Tor H_2(\bar{B})=E^\infty_{1,1}.
\]
Putting together Lemmas \ref{lem:spectralpage2},\ref{lem:stabilizer-homology} and \ref{lem:spectral-sequence-stabilization} thus identifies $H_2(\bar{B})$, and Lemma \ref{lem:borel-to-total-space} finishes the proposition.
\end{proof}

\subsection{Main proof}

\begin{proof}[Proof of Theorem \ref{t:BHE3foldsrank0}] 
We first claim $H_*(M)$ is torsion-free.  By our assumption of coprimality of the divisor subgroups, it follows from Lemmas \ref{l:isolatedcyclicity} and \ref{lem:cyclic-intersection-stabilizer} that $B$ is a cyclic orbifold, and indeed all the hypotheses of Proposition \ref{p:torsion-formula} hold.  Since by assumption every orbifold divisor \(D_i\) has genus zero, it follows from Proposition \ref{p:torsion-formula} that $H_2(M)$ is torsion-free.  Then $H_*(M)$ is torsion-free by Lemma \ref{lem:torsionfree}.  Using this we prove that the cup product vanishes on $H^2(M)$.  Working first with real coefficients, and using that $H^*_{\mathrm{orb}}(B;\mathbb R) \cong H^*(\bar{B},; \mathbb R)$, it follows from Lemma \ref{lem:borel-to-total-space} that 
\[
\pi^* \colon H^2_{\mathrm{orb}}(B;\mathbb{R}) \longrightarrow H^2(M;\mathbb{R})
\]
is surjective.  As the orbifold Euler class is nonvanishing, the argument of Lemma \ref{lem:p1-vanishing} therefore shows that $\pi^* \colon H^4_{\mathrm{orb}}(B;\mathbb{R}) \longrightarrow H^4(M;\mathbb{R})$ is the zero map, and it then follows easily that the cup product vanishes on $H^2(M;\mathbb{R})$. Since $H_*(M)$ is torsion-free, the cup product vanishes on
$H^2(M)$ as well.  Noting finally that $M$ is spin by Lemma \ref{l:spin}, and $p_1(M)$ vanishes by Lemma \ref{lem:p1-vanishing}, the result follows from the Wall-Jupp classification \cite{jupp1973classification,wall1966classification}. 
\end{proof}


\end{document}